\documentclass{amsart}
\usepackage{amsfonts}
\usepackage{graphicx}
\usepackage{amscd}

\newtheorem{theorem}{Theorem}[section]
\theoremstyle{plain}
\newtheorem*{Jor}{J\"{o}rgens' Theorem}
\newtheorem*{Lagrange}{Lagrange's Identity}

\newtheorem{corollary}[theorem]{Corollary}

\newtheorem{lemma}[theorem]{Lemma}

\theoremstyle{remark}
\newtheorem{remark}[theorem]{Remark}
\theoremstyle{example}

\numberwithin{equation}{section}

\input{tcilatex}

\begin{document}
\title[Minimal graphs]{Minimal graphs in $\mathbb{R}^{4}$ with bounded
Jacobians}
\author{Th. Hasanis}
\address{Department of Mathematics, University of Ioannina, 45110, Ioannina,
Greece}
\email{thasanis@uoi.gr, ansavas@cc.uoi.gr, tvlachos@uoi.gr}
\author{A. Savas-Halilaj}
\author{Th. Vlachos}
\subjclass{Primary 53C42}
\keywords{Minimal surface, Bernstein type theorem, Jacobian.}
\thanks{The second author is supported financially from the ''Foundation for
Education and European Culture''.}

\begin{abstract}
We obtain a Bernstein type result for entire two dimensional minimal graphs
in $\mathbb{R}^{4}$, which extends a previous one due to L. Ni.
Moreover, we provide a characterization for complex analytic curves.
\end{abstract}

\maketitle

\section{Introduction}

The famous theorem of Bernstein states that the only entire minimal graphs
in the Euclidean space $\mathbb{R}^{3}$ are planes. More precisely, if $f:%
\mathbb{R}^{2}\rightarrow \mathbb{R}$ is an entire (i.e., defined over all
of $\mathbb{R}^{2}$) smooth function whose graph%
\begin{equation*}
G_{f}:=\left\{ \left( x,y,f\left( x,y\right) \right) \in \mathbb{R}%
^{3}:\left( x,y\right) \in \mathbb{R}^{2}\right\}
\end{equation*}%
is a minimal surface, then it is an affine function, and the graph is a
plane.

This type of result has been generalized in higher dimension and codimension
under various conditions. See \cite{EH}, \cite{FC}, \cite{S} and the
references therein for the codimension one case and \cite{JX}, \cite{W}, 
\cite{Y} for the higher codimension case.

The aim of this paper is to study the following special case. Let $M$ be a
minimal surface in $\mathbb{R}^{4}$ that can be described as the graph of an
entire and smooth vector valued function $f:\mathbb{R}^{2}\rightarrow 
\mathbb{R}^{2},$ $f\left( x,y\right) =\left( f_{1}\left( x,y\right)
,f_{2}\left( x,y\right) \right) $, that is%
\begin{equation*}
M=G_{f}:=\left\{ \left( x,y,f_{1}\left( x,y\right) ,f_{2}\left( x,y\right)
\right) \in \mathbb{R}^{4}:\left( x,y\right) \in \mathbb{R}^{2}\right\} .
\end{equation*}%
The following question arises in a natural way: \textit{Is it true that the
graph }$G_{f}$\textit{\ of }$f$\textit{\ is a plane in }$\mathbb{R}^{4}$? In
general, the answer is no. An easy counterexample is given by the function $%
f\left( x,y\right) =\left( x^{2}-y^{2},2xy\right) $. Actually, the graph of
any holomorphic function $\Phi :\mathbb{C\rightarrow C}$ is a minimal
surface. So, the problem of finding geometric conditions in order to have a
result of Bernstein type is reasonable. In this direction, R. Schoen \cite%
{Sc} obtained a Bernstein type result by imposing the assumption that $f:%
\mathbb{R}^{2}\rightarrow \mathbb{R}^{2}$ is a diffeomorphism. Moreover, L.
Ni \cite{N} by using the result of R. Schoen \cite{Sc} and results due to J.
Wolfson \cite{Wo} on minimal Lagrangian surfaces has derived a result of
Bernstein type under the assumption that $f$ is an \textit{area-preserving
map}, that is the \textit{Jacobian} $J_{f}:=\det \left( df\right) $
satisfies $J_{f}=1$, where $df$ denotes the differential of $f$.

In this paper we prove, firstly, the following result of Bernstein type,
which generalizes the result due to L. Ni.

\begin{theorem}
Let $f:\mathbb{R}^{2}\mathbb{\rightarrow R}^{2}$ be an entire smooth vector
valued function such that its graph $G_{f}$ is a minimal surface in $\mathbb{%
R}^{4}$. If the Jacobian $J_{f}$ of $f$ is bounded, then $G_{f}$ is a plane.
\end{theorem}

As a consequence we derive an easy alternative proof of the following well
known result of J\"{o}rgens \cite{J}:

\begin{Jor}
\textit{The only entire solutions }$f:\mathbb{R}^{2}\rightarrow \mathbb{R}$%
\textit{\ of the Monge-Ampere equation }$f_{xx}f_{yy}-f_{xy}^{2}=1$\textit{\
are the quadratic polynomials.}
\end{Jor}

There are plenty of entire minimal graphs in $\mathbb{R}^{4}$, the so called 
\textit{complex analytic curves}. More precisely, if $\Phi :\mathbb{%
C\rightarrow C}$ is any entire holomorphic or anti-holomorphic function,
then the graph 
\begin{equation*}
G_{\Phi }=\left\{ \left( z,\Phi \left( z\right) \right) \in \mathbb{C}%
^{2}:z\in \mathbb{C}\right\}
\end{equation*}%
of $\Phi $ in $\mathbb{C}^{2}=\mathbb{R}^{4}$ is a minimal surface and is
called a complex analytic curve. Such surfaces are locally characterized
(see for example L.P Eisenhart \cite{E}) by the relation $\left\vert
K\right\vert =\left\vert K_{N}\right\vert $, where $K$ and $K_{N}$ stand for
the Gauss and normal curvature of the surface, respectively. The following
result is in valid.

\begin{theorem}
Let $G_{f}$ be the graph of an entire smooth vector valued function $f:%
\mathbb{R}^{2}\mathbb{\rightarrow R}^{2}$ with Gaussian curvature $K$ and
normal curvature$\ K_{N}$. Assume that $G_{f}$ is minimal in $\mathbb{R}^{4}$%
. Then,%
\begin{equation*}
\inf\limits_{K<0}\frac{\left\vert K_{N}\right\vert }{\left\vert K\right\vert 
}=0,
\end{equation*}%
unless $G_{f}$ is a complex analytic curve.
\end{theorem}

\begin{corollary}
Let $G_{f}$ be the graph of an entire smooth vector valued function $f:%
\mathbb{R}^{2}\mathbb{\rightarrow R}^{2}$ with Gaussian curvature $K$ and
normal curvature$\ K_{N}$. If $K_{N}=cK$, where $c$ is a constant, then $%
G_{f}$ is a complex analytic curve. More precisely, $K_{N}=K=0$ and $G_{f}$
is a plane or $\left\vert c\right\vert =1$ and $G_{f}$ is a non-trivial
complex analytic curve.
\end{corollary}

\section{Basic Notation and Definitions}

A surface $M$ in the Euclidean space $\mathbb{R}^{n}$ is represented,
locally, by a transformation $X:D\subset \mathbb{R}^{2}\rightarrow \mathbb{R}%
^{n}$ of rank 2, given by 
\begin{equation*}
X\left( x,y\right) =\left( f_{1}\left( x,y\right) ,f_{2}\left( x,y\right)
,\cdots ,f_{n}\left( x,y\right) \right) ,\text{ \ \ \ }\left( x,y\right) \in
D,
\end{equation*}%
where $D$ is an open subset of $\mathbb{R}^{2}$ and $f_{i}:D\rightarrow 
\mathbb{R}$, $i\in \left\{ 1,\cdots ,n\right\} ,$ are smooth functions.
Following the standard notation of differential geometry, we denote by $%
\left\langle \text{ },\text{ }\right\rangle $ the Euclidean inner product on 
$\mathbb{R}^{n}$ and by $E$, $F$, $G$ the \textit{coefficients of the first
fundamental form,} which are given by 
\begin{equation*}
E=\sum_{i=1}^{n}\left( \frac{\partial f_{i}}{\partial x}\right) ^{2}\text{, }%
F=\sum_{i=1}^{n}\frac{\partial f_{i}}{\partial x}\frac{\partial f_{i}}{%
\partial y}\text{, }G=\sum_{i=1}^{n}\left( \frac{\partial f_{i}}{\partial y}%
\right) ^{2}.
\end{equation*}%
We recall that the parameters $\left( x,y\right) $ are called \textit{%
isothermal} if and only if $E=G$ and $F=0$, everywhere on $D$.

Consider a local orthonormal frame field $\left\{ e_{1},e_{2};\xi
_{3},\cdots ,\xi _{n}\right\} $ in $\mathbb{R}^{n}$ such that, restricted to 
$M$, the vectors $e_{1},e_{2}$ are tangent to $M$ and, consequently, $\xi
_{3},\cdots ,\xi _{n}$ are normal to $M$. Denote by $\overline{\nabla }$ the
usual linear connection on $\mathbb{R}^{n}$ and let 
\begin{equation*}
h_{ij}^{\alpha }=\left\langle \overline{\nabla }_{e_{i}}\xi _{\alpha
},e_{j}\right\rangle ,\quad i,j\in \left\{ 1,2\right\} ,\quad \alpha \in
\left\{ 3,\cdots ,n\right\} ,
\end{equation*}%
be the \textit{coefficients of the second fundamental form}.

The \textit{mean curvature vector} $H$ and the \textit{Gauss curvature} $K$
of $M$ are given, respectively, by%
\begin{eqnarray*}
H &=&\frac{1}{2}\dsum\limits_{\alpha =3}^{n}\left( h_{11}^{\alpha
}+h_{22}^{\alpha }\right) \xi _{\alpha }, \\
K &=&\dsum\limits_{\alpha =3}^{n}\left( h_{11}^{\alpha }h_{22}^{\alpha
}-\left( h_{12}^{\alpha }\right) ^{2}\right) .
\end{eqnarray*}%
Moreover, if%
\begin{equation*}
\left\vert h\right\vert ^{2}=\dsum\limits_{i,j=1}^{2}\dsum\limits_{\alpha
=3}^{n}\left( h_{ij}^{\alpha }\right) ^{2}
\end{equation*}%
is the square of the length of the second fundamental form $h$, then the
Gauss equation implies%
\begin{equation*}
2K=4H^{2}-\left\vert h\right\vert ^{2}.
\end{equation*}%
In the case where $M$ is \textit{minimal}, i.e., $H=0$, the above become%
\begin{eqnarray}
K &=&-\dsum\limits_{\alpha =3}^{n}\left\{ \left( h_{11}^{\alpha }\right)
^{2}+\left( h_{12}^{\alpha }\right) ^{2}\right\} , \\
2K &=&-\left\vert h\right\vert ^{2}.
\end{eqnarray}%
Another geometric invariant which plays an important role in the theory of
surfaces in $\mathbb{R}^{4}$ is the \textit{normal curvature} $K_{N}$ of $M$
which is given by%
\begin{equation*}
K_{N}=\dsum\limits_{i=1}^{2}\left(
h_{i1}^{3}h_{2i}^{4}-h_{2i}^{3}h_{1i}^{4}\right) .
\end{equation*}%
In particular, for minimal surfaces we have%
\begin{equation}
K_{N}=2\left( h_{11}^{3}h_{12}^{4}-h_{12}^{3}h_{11}^{4}\right) .
\end{equation}

One of the simplest ways to express a surface in $\mathbb{R}^{n+2}$ is in 
\textit{non-parametric form}, that is to say, as the graph%
\begin{equation*}
G_{f}=\left\{ \left( x,y,f_{1}\left( x,y\right) ,\cdots ,f_{n}\left(
x,y\right) \right) \in \mathbb{R}^{n+2}:\left( x,y\right) \in D\right\}
\end{equation*}%
of a vector valued map $f:D\rightarrow \mathbb{R}^{n}$, $f\left( x,y\right)
=\left( f_{1}\left( x,y\right) ,\cdots ,f_{n}\left( x,y\right) \right) $,
where $D$ is an open subset of $\mathbb{R}^{2}$. Of course, any surface can
be locally described in this manner. By computing the Euler-Lagrange
equations for the area integral we see that the surface $G_{f}$ is minimal
if and only if $f$ satisfies the following equation,%
\begin{equation}
\left( 1+\left\vert f_{y}\right\vert ^{2}\right) f_{xx}-2\left\langle
f_{x},f_{y}\right\rangle f_{xy}+\left( 1+\left\vert f_{x}\right\vert
^{2}\right) f_{yy}=0.
\end{equation}%
This is the classical \textit{non-parametric minimal surface equation}.

The following result due to R. Osserman \cite[Theorem 5.1]{O} is the main
tool for the proofs of our results.

\begin{theorem}
Let $f:\mathbb{R}^{2}\rightarrow \mathbb{R}^{2}$ be an entire solution of
the minimal surface equation. Then there exists real constants $a$, $b,$
with $b>0$, and a non-singular linear transformation 
\begin{equation*}
x=u,\quad y=au+bv,
\end{equation*}%
such that $\left( u,v\right) $ are global isothermal parameters for the
surface $G_{f}$.
\end{theorem}

Moreover the following identity is useful in the proofs.

\begin{Lagrange}
For two vectors $V=\left( v_{1},\cdots ,v_{n}\right) $, $W=\left( \omega
_{1},\cdots ,\omega _{n}\right) $ in $\mathbb{R}^{n}$ we have%
\begin{equation}
\left( \dsum\limits_{i=1}^{n}v_{i}^{2}\right) \left(
\dsum\limits_{i=1}^{n}\omega _{i}^{2}\right) -\left(
\dsum\limits_{i=1}^{n}v_{i}\omega _{i}\right) ^{2}=\dsum\limits_{i<j}\left(
v_{i}\omega _{j}-v_{j}\omega _{i}\right) ^{2}.
\end{equation}
\end{Lagrange}

\section{Proofs of the results}

Let $f:\mathbb{R}^{2}\rightarrow \mathbb{R}^{2}$, $f\left( x,y\right)
=\left( f_{1}\left( x,y\right) ,f_{2}\left( x,y\right) \right) $, $\left(
x,y\right) \in \mathbb{R}^{2}$, be an entire solution of the minimal surface
equation. Then, the graph 
\begin{equation*}
G_{f}=\left\{ \left( x,y,f_{1}\left( x,y\right) ,f_{2}\left( x,y\right)
\right) \in \mathbb{R}^{4}:\left( x,y\right) \in \mathbb{R}^{2}\right\}
\end{equation*}%
of $f$ is a minimal surface in $\mathbb{R}^{4}$. By virtue of the Theorem
2.1, we can introduce global isothermal parameters $\left( u,v\right) $, via
a non-singular transformation 
\begin{equation*}
x=u,\quad y=au+bv,
\end{equation*}%
where $a$, $b$ are real constants with $b>0$. Now, the minimal surface $%
G_{f} $ is parametrized via the map 
\begin{equation*}
X\left( u,v\right) =\left( u,au+bv,\varphi \left( u,v\right) ,\psi \left(
u,v\right) \right) ,
\end{equation*}%
where $\varphi \left( u,v\right) :=f_{1}\left( u,au+bv\right) $ and $\psi
\left( u,v\right) :=f_{2}\left( u,au+bv\right) $. Since $\left( u,v\right) $
are isothermal parameters, the vectors 
\begin{equation}
X_{u}=\left( 1,a,\varphi _{u},\psi _{u}\right) ,\quad X_{v}=\left(
0,b,\varphi _{v},\psi _{v}\right)
\end{equation}%
are orthogonal and of the same length, that is%
\begin{equation}
\begin{array}{l}
\varphi _{u}\varphi _{v}+\psi _{u}\psi _{v}=-ab, \\ 
\\ 
E=1+a^{2}+\varphi _{u}^{2}+\psi _{u}^{2}=b^{2}+\varphi _{v}^{2}+\psi
_{v}^{2}.%
\end{array}%
\end{equation}%
Moreover, the fact that the surface $G_{f}$ is minimal, implies that\ the
functions $\varphi $ and $\psi $ are harmonic, that is 
\begin{equation}
\varphi _{uu}+\varphi _{vv}=0,\quad \psi _{uu}+\psi _{vv}=0.
\end{equation}%
Appealing to the Lagrange's Identity and taking the relations $\left(
3.2\right) $ into account, we obtain 
\begin{eqnarray*}
E^{2} &=&b^{2}+\varphi _{v}^{2}+\psi _{v}^{2}+\left( a\varphi _{v}-b\varphi
_{u}\right) ^{2} \\
&&+\left( a\psi _{v}-b\psi _{u}\right) ^{2}+\left( \varphi _{u}\psi
_{v}-\varphi _{v}\psi _{u}\right) ^{2},
\end{eqnarray*}%
or equivalently, 
\begin{equation}
E^{2}=\left( 1+a^{2}+b^{2}\right) E-b^{2}+\left( \varphi _{u}\psi
_{v}-\varphi _{v}\psi _{u}\right) ^{2}.
\end{equation}%
We set $\Phi \left( u,v\right) =\left( \varphi \left( u,v\right) ,\psi
\left( u,v\right) \right) $. Because of the relation 
\begin{equation*}
\frac{\partial \left( \varphi ,\psi \right) }{\partial \left( u,v\right) }=%
\frac{\partial \left( f_{1},f_{2}\right) }{\partial \left( x,y\right) }\frac{%
\partial \left( x,y\right) }{\partial \left( u,v\right) }
\end{equation*}%
for the Jacobians, we have 
\begin{equation*}
J_{\Phi }=bJ_{f},
\end{equation*}%
where $J_{f}$, $J_{\Phi }$ stand for the Jacobians of $f$ and $\Phi ,$
respectively. So $\left( 3.4\right) $ becomes 
\begin{equation}
J_{\Phi }^{2}=E^{2}-\left( 1+a^{2}+b^{2}\right) E+b^{2},
\end{equation}%
a useful identity for us.

Now we are ready to give the proof of Theorem 1.1.

\begin{proof}[Proof of Theorem 1.1]
Since the Jacobian $J_{\Phi }$ is bounded, we conclude from $\left(
3.5\right) $ that $E$, and thus $\log E$, is bounded from above. On the
other hand, the Gaussian curvature $K$ of $G_{f}$ is given by%
\begin{equation*}
K=-\frac{\Delta \log E}{2E},
\end{equation*}%
where $\Delta $ is the usual Laplacian operator on the $\left( u,v\right) $%
-plane. Appealing to $\left( 2.1\right) $, we deduce that the Gaussian
curvature $K$ is non-positive. Consequently, $\Delta \log E\geq 0$ and thus
the function $\log E$ is a subharmonic function defined on the whole plane.
Since $\log E$ is also bounded from above, we deduce that $E$ is constant
and consequently $K$ is identically zero. Then it follows immediately from $%
\left( 2.2\right) $ that the graph $G_{f}$ of $f$ is totally geodesic and
hence a plane.
\end{proof}

\begin{remark}
In a similar way, we can prove the following result: Let $f:\mathbb{R}%
^{2}\rightarrow \mathbb{R}^{n}$,%
\begin{equation*}
f\left( x,y\right) =\left( f_{1}\left( x,y\right) ,f_{2}\left( x,y\right)
,\cdots ,f_{n}\left( x,y\right) \right) ,
\end{equation*}%
be a vector valued function, defined on the whole $\mathbb{R}^{2}$, which is
a solution of the minimal surface equation. If the quantity%
\begin{equation*}
\dsum\limits_{i<j}\left( \frac{\partial f_{i}}{\partial x}\frac{\partial
f_{j}}{\partial y}-\frac{\partial f_{j}}{\partial x}\frac{\partial f_{i}}{%
\partial y}\right) ^{2}
\end{equation*}%
is bounded, then the graph $G_{f}$ of $f$ is a plane in $\mathbb{R}^{n+2}$.
\end{remark}

Now, we show that one can get the well known J\"{o}rgens' result \cite{J} as
a consequence of Theorem 1.1.

\begin{proof}[Proof of J\"{o}rgens' Theorem]
Obviously $f_{xx}+f_{yy}\neq 0$ everywhere on $\mathbb{R}^{2}$. We consider
the function $\Theta :\mathbb{R}^{2}\rightarrow \mathbb{R}$, given by%
\begin{equation*}
\Theta =\frac{f_{xx}f_{yy}-f_{xy}^{2}-1}{f_{xx}+f_{yy}}.
\end{equation*}%
The function $\Theta $, thanks to our assumption, is identically zero, so $%
\Theta _{x}=\Theta _{y}=0$. On the other hand, one can readily verify that
the equations $\Theta _{x}=\Theta _{y}=0$ are equivalent to the minimal
surface equation for the function $g:\mathbb{R}^{2}\rightarrow \mathbb{R}%
^{2} $, defined by $g\left( x,y\right) =\left( f_{x}\left( x,y\right)
,f_{y}\left( x,y\right) \right) $. Moreover, we have $J_{g}=1$. So,
according to Theorem 1.1, the graph $G_{g}$ of $g$ is a plane and the result
is immediate.
\end{proof}

For the proof of Theorem 1.2, we need the following auxiliary result.

\begin{lemma}
Let $\Phi :\mathbb{R}^{2}\rightarrow \mathbb{R}^{2}$, $\Phi \left(
u,v\right) =\left( \varphi \left( u,v\right) ,\psi \left( u,v\right) \right)
,$ be a map, where $\varphi $ and $\psi $ are harmonic functions on $\mathbb{%
R}^{2}$, i.e., a harmonic map. Then $\inf \left\vert J_{\Phi }\right\vert =0$%
, unless $\Phi $ is an affine map.
\end{lemma}

\begin{proof}
Suppose in the contrary that $\Phi $ is not affine and $\inf \left\vert
J_{\Phi }\right\vert =c>0.$ Hence $\left\vert J_{\Phi }\right\vert \geq c>0$%
. Assume at first that $J_{\Phi }\geq c>0$. We view $\Phi $ as a complex
valued function $\Phi :\mathbb{C\rightarrow C}$, $\Phi =\varphi +i\psi $.
Then, for $z=u+iv$, we have%
\begin{equation*}
\Phi _{z}=\frac{1}{2}\left( \varphi _{u}+\psi _{v}\right) +\frac{i}{2}\left(
\psi _{u}-\varphi _{v}\right)
\end{equation*}%
and%
\begin{equation*}
\Phi _{\overline{z}}=\frac{1}{2}\left( \varphi _{u}-\psi _{v}\right) +\frac{i%
}{2}\left( \psi _{u}+\varphi _{v}\right) .
\end{equation*}%
A simple calculation shows that%
\begin{equation}
J_{\Phi }=\left\vert \Phi _{z}\right\vert ^{2}-\left\vert \Phi _{\overline{z}%
}\right\vert ^{2}.
\end{equation}%
Furthermore, since $\varphi $ and $\psi $ are harmonic functions it follows
that the function $\Phi _{z}$ is holomorphic. From our assumption and $%
\left( 3.6\right) $ we get%
\begin{equation}
\left\vert \Phi _{z}\right\vert ^{2}\geq \left\vert \Phi _{\overline{z}%
}\right\vert ^{2}+c\geq c>0.
\end{equation}%
Since $\Phi _{z}$ is an entire holomorphic function, Picard's Theorem
implies that $\Phi _{z}$ must be constant. Therefore, there are real
constants$\ \kappa ,\lambda $ such that%
\begin{equation*}
\varphi _{u}+\psi _{v}=2\kappa \text{ \ and \ }\psi _{u}-\varphi
_{v}=2\lambda .
\end{equation*}%
Then from $\left( 3.7\right) $ we deduce that%
\begin{equation*}
\left( \psi _{v}-2\kappa \right) ^{2}+\left( \psi _{u}-2\lambda \right)
^{2}\leq \kappa ^{2}+\lambda ^{2}-c.
\end{equation*}%
By the harmonicity of the functions $\psi _{v}-2\kappa $, $\psi
_{u}-2\lambda $ and the Liouville's Theorem, we deduce that $\varphi $ and $%
\psi $ are affine functions, which contradicts our assumptions.

Assume now that $J_{\Phi }\leq -c<0$. In this case, we consider the complex
valued function $\widetilde{\Phi }=\psi +i\varphi $. Since $J_{\widetilde{%
\Phi }}=-J_{\Phi }\geq c>0$, proceeding as above we deduce that $\widetilde{%
\Phi }$ is affine, and consequently $\Phi $ is affine. This is again a
contradiction. Thus $\inf \left\vert J_{\Phi }\right\vert =0$, and the proof
is concluded.
\end{proof}

\begin{proof}[Proof of Theorem 1.2]
Assume that $G_{f}$ is not a plane and that%
\begin{equation*}
\inf_{K<0}\frac{\left\vert K_{N}\right\vert }{\left\vert K\right\vert }>0.
\end{equation*}%
We introduce global isothermal parameters $\left( u,v\right) $ such that the
minimal surface $G_{f}$ is parametrized via the map%
\begin{equation*}
X\left( u,v\right) =\left( u,au+bv,\varphi \left( u,v\right) ,\psi \left(
u,v\right) \right) ,
\end{equation*}%
where $a,b$ are real constants with $b>0$.

We claim that $\left( a,b\right) =\left( 0,1\right) $. Arguing indirectly,
we assume that $\left( a,b\right) \neq \left( 0,1\right) $. Differentiating $%
\left( 3.2\right) $ with respect to $u,v$ and taking $\left( 3.3\right) $
into account, we find%
\begin{equation}
\begin{array}{l}
\varphi _{uu}\varphi _{v}+\varphi _{u}\varphi _{uv}=-\psi _{uu}\psi
_{v}-\psi _{u}\psi _{uv}, \\ 
\\ 
\varphi _{uu}\varphi _{u}-\varphi _{v}\varphi _{uv}=-\psi _{uu}\psi
_{u}+\psi _{v}\psi _{uv}.%
\end{array}%
\end{equation}%
Squaring both of them and summing we obtain%
\begin{equation}
\left( \varphi _{u}^{2}+\varphi _{v}^{2}\right) \left( \varphi
_{uu}^{2}+\varphi _{uv}^{2}\right) =\left( \psi _{u}^{2}+\psi
_{v}^{2}\right) \left( \psi _{uu}^{2}+\psi _{uv}^{2}\right) .
\end{equation}%
Consider the following subset of $\mathbb{R}^{2}$%
\begin{equation*}
M_{0}=\left\{ \left( u,v\right) \in \mathbb{R}^{2}:\omega \left( u,v\right)
=0\right\} ,
\end{equation*}%
where 
\begin{equation*}
\omega \left( u,v\right) :=\left( \varphi _{u}^{2}+\varphi _{v}^{2}\right)
\left( \varphi _{uu}^{2}+\varphi _{uv}^{2}\right) ,
\end{equation*}%
or, equivalently, in view of $\left( 3.9\right) $%
\begin{equation*}
\omega \left( u,v\right) =\left( \psi _{u}^{2}+\psi _{v}^{2}\right) \left(
\psi _{uu}^{2}+\psi _{uv}^{2}\right) .
\end{equation*}%
We claim that the complement $M_{1}=\mathbb{R}^{2}-M_{0}$ is dense in $%
\mathbb{R}^{2}$. To this purpose it is enough to show that the interior, $%
\limfunc{int}\left( M_{0}\right) $, of $M_{0}$ is empty. Assume in the
contrary that $\limfunc{int}\left( M_{0}\right) \neq \emptyset $ and let $U$
be a connected component of $\limfunc{int}\left( M_{0}\right) $. Then it
follows easily that the analytic functions $\varphi $ and $\psi $ are
affine. Thus, by analyticity, $G_{f}$ is a plane, which is a contradiction.

In the sequel, we work on $M_{1}$. By virtue of $\left( 3.8\right) $, we get 
\begin{equation}
\begin{array}{l}
\varphi _{uv}=\dfrac{-\left( \varphi _{u}\psi _{u}+\varphi _{v}\psi
_{v}\right) \psi _{uv}-J_{\Phi }\psi _{uu}}{\varphi _{u}^{2}+\varphi _{v}^{2}%
}, \\ 
\\ 
\varphi _{uu}=\dfrac{J_{\Phi }\psi _{uv}-\left( \varphi _{u}\psi
_{u}+\varphi _{v}\psi _{v}\right) \psi _{uu}}{\varphi _{u}^{2}+\varphi
_{v}^{2}}.%
\end{array}%
\end{equation}%
The vector fields%
\begin{equation*}
\xi =\left( -b\varphi _{u}+a\varphi _{v},-\varphi _{v},b,0\right) ,\text{%
\quad }\eta =\left( -b\psi _{u}+a\psi _{v},-\psi _{v},0,b\right)
\end{equation*}%
are normal to $G_{f}$ and satisfy%
\begin{equation*}
\left\vert \xi \right\vert ^{2}\left\vert \eta \right\vert ^{2}-\left\langle
\xi ,\eta \right\rangle ^{2}=b^{2}E^{2}.
\end{equation*}%
We, easily, check that the vector fields $\left\{ e_{1},e_{2};\xi _{3},\xi
_{4}\right\} $ given by%
\begin{equation*}
\begin{array}{ll}
e_{1}=\dfrac{1}{\sqrt{E}}X_{u}, & e_{2}=\dfrac{1}{\sqrt{E}}X_{v}, \\ 
\xi _{3}=\dfrac{1}{\left\vert \xi \right\vert }\xi , & \xi _{4}=\dfrac{1}{%
b\left\vert \xi \right\vert E}\left( \left\vert \xi \right\vert ^{2}\eta
-\left\langle \xi ,\eta \right\rangle \xi \right) ,%
\end{array}%
\end{equation*}%
constitute an orthonormal frame field along $G_{f}$. Moreover, $\xi _{3}$
and $\xi _{4}$ are normal to $G_{f}$. Then a straightforward computation
shows that the coefficients of the second fundamental form are given by%
\begin{equation*}
\begin{array}{ll}
h_{11}^{3}=-\dfrac{b\varphi _{uu}}{E\left\vert \xi \right\vert }, & 
h_{11}^{4}=\dfrac{\left\langle \xi ,\eta \right\rangle \varphi
_{uu}-\left\vert \xi \right\vert ^{2}\psi _{uu}}{E^{2}\left\vert \xi
\right\vert }, \\ 
h_{12}^{3}=-\dfrac{b\varphi _{uv}}{E\left\vert \xi \right\vert }, & 
h_{12}^{4}=\dfrac{\left\langle \xi ,\eta \right\rangle \varphi
_{uv}-\left\vert \xi \right\vert ^{2}\psi _{uv}}{E^{2}\left\vert \xi
\right\vert }.%
\end{array}%
\end{equation*}%
So using $\left( 2.1\right) $ and $\left( 2.3\right) $ and $\left(
3.10\right) $, we find%
\begin{equation}
K=\frac{1}{E^{3}}\frac{\psi _{uu}^{2}+\psi _{uv}^{2}}{\varphi
_{u}^{2}+\varphi _{v}^{2}}\left( 2b^{2}-\left( 1+a^{2}+b^{2}\right) E\right)
\end{equation}%
and%
\begin{equation}
K_{N}=\frac{2b}{E^{3}}\frac{\psi _{uu}^{2}+\psi _{uv}^{2}}{\varphi
_{u}^{2}+\varphi _{v}^{2}}J_{\Phi }.
\end{equation}

The second equation of $\left( 3.2\right) $, yields%
\begin{equation*}
E\geq \frac{1+a^{2}+b^{2}}{2}.
\end{equation*}%
Hence,%
\begin{equation*}
2b^{2}-\left( 1+a^{2}+b^{2}\right) E\leq -\frac{1}{2}\left( a^{2}+\left(
b-1\right) ^{2}\right) \left( a^{2}+\left( b+1\right) ^{2}\right) <0.
\end{equation*}%
This shows that%
\begin{equation*}
M_{1}\subset \left\{ \left( u,v\right) \in \mathbb{R}^{2}:K\left( u,v\right)
<0\right\} .
\end{equation*}%
Moreover,%
\begin{equation*}
\frac{K_{N}^{2}}{K^{2}}=4b^{2}W\left( E\right) ,
\end{equation*}%
where $W\left( t\right) $ is the increasing real valued function%
\begin{equation*}
W\left( t\right) :=\frac{t^{2}-\left( 1+a^{2}+b^{2}\right) t+b^{2}}{\left(
\left( 1+a^{2}+b^{2}\right) t-2b^{2}\right) ^{2}},\text{ \ \ \ \ }t\geq 1.
\end{equation*}%
From our assumption $\inf_{K<0}\frac{\left\vert K_{N}\right\vert }{%
\left\vert K\right\vert }>0$, we get 
\begin{equation*}
\inf_{M_{1}}\frac{\left\vert K_{N}\right\vert }{\left\vert K\right\vert }>0%
\text{.}
\end{equation*}%
Since $W\left( t\right) $ is increasing, we have%
\begin{equation*}
\inf_{M_{1}}\frac{K_{N}^{2}}{K^{2}}=4b^{2}W\left( \inf_{M_{1}}E\right) .
\end{equation*}%
Hence $W\left( \inf_{M_{1}}E\right) >0$ or, equivalently,%
\begin{equation*}
\left( \inf_{M_{1}}E\right) ^{2}-\left( 1+a^{2}+b^{2}\right)
\inf_{M_{1}}E+b^{2}>0\text{.}
\end{equation*}%
Appealing to the identity $\left( 3.5\right) $, we deduce that $%
\inf_{M_{1}}\left\vert J_{\Phi }\right\vert >0$. By continuity, and bearing
in mind the fact that $M_{1}$ is dense in $\mathbb{R}^{2}$, we infer that $%
\left\vert J_{\Phi }\right\vert $ is bounded from below away from zero. On
the other hand $\Phi \left( u,v\right) =\left( \varphi \left( u,v\right)
,\psi \left( u,v\right) \right) $ is a harmonic map. Therefore, according to
Lemma 3.2, $G_{f}$ is a plane which contradicts our assumptions. Thus $%
\left( a,b\right) =\left( 0,1\right) $ and the equations $\left( 3.2\right) $
become%
\begin{equation*}
\begin{array}{l}
\varphi _{u}\varphi _{v}+\psi _{u}\psi _{v}=0, \\ 
\\ 
\varphi _{u}^{2}+\psi _{u}^{2}=\varphi _{v}^{2}+\psi _{v}^{2}.%
\end{array}%
\end{equation*}%
So, $\varphi _{u}=\pm \psi _{v}$, $\varphi _{v}=\mp \psi _{u}$ and $G_{f}$
is a complex analytic curve.
\end{proof}

\begin{proof}[Proof of Corollary 1.3]
In the case where $K\equiv 0$, $G_{f}$ is a plane. Consider now the case
where $K$ is not identically zero. According to Theorem 1.2 we have $c=0,$
unless $G_{f}$ is a complex analytic curve. We claim that the case $c=0$\
does not occur. Indeed, arguing indirectly suppose that $c=0$. As in the
proof of Theorem 1.2, the set $M_{1}$ is dense in $\mathbb{R}^{2}$. From the
assumption $K_{N}=cK$, we conclude that $K_{N}=0$ in $M_{1}$. Furthermore,
the relation $\left( 3.12\right) $ yields that $J_{\Phi }=0$. Taking into
account the identity $\left( 3.5\right) ,$ we get that $E$ is constant,
which implies that $K$ is identically zero, a contradiction. Therefore, $%
G_{f}$ is a complex analytic curve.
\end{proof}

\end{document}